\documentclass{amsart}
\usepackage{amsmath}
\usepackage{amsthm}
\usepackage{amssymb}
\usepackage{amsfonts}
\usepackage{epsfig}
\usepackage{setspace}
\usepackage{color}

\newtheorem{theorem}{Theorem}[section]

\newtheorem{lm}[theorem]{Lemma}
\newtheorem{tr}[theorem]{Theorem}

\newtheorem{cor}[theorem]{Corollary}

\newtheorem{rem}[theorem]{Remark}
\newtheorem{pr}[theorem]{Proposition}
\newtheorem{quest}[theorem]{Problem}
\newtheorem{ex}[theorem]{Example}

\begin{document}
\title[Minimal degrees of invariants of (super)groups]{Minimal degrees of invariants of (super)groups - a connection to cryptology}
\thanks{This publication was made possible by a NPRF award NPRP 6 - 1059 - 1 - 208 from the Qatar National Research Fund (a member of The Qatar Foundation). The statements made herein are solely the responsibility of the authors.}
\author{Franti\v sek~ Marko}
\email{fxm13@psu.edu}
\address{Penn State Hazleton, 76 University Drive, Hazleton, PA 18202, USA}
\author{Alexandr N. Zubkov}
\email{a.zubkov@yahoo.com}
\address{Sobolev Institute of Mathematics, Siberian Branch of Russian Academy of Science (SORAN), Omsk, Pevtzova 13, 644043, Russia}
\begin{abstract} 
We investigate questions related to the minimal degree of invariants of finitely generated diagonalizable groups. These questions were raised in connection to security of a public key cryptosystem based on invariants of diagonalizable groups. We derive results for minimal degrees of invariants of finite groups, abelian groups and algebraic groups. For algebraic groups we relate the minimal degree of the group to the minimal degrees of its tori.
Finally, we investigate invariants of certain supergroups that are superanalogs of tori. It is interesting to note that a basis of these invariants is not given by monomials.

\end{abstract}
\maketitle
\section*{Introduction}
Let $G$ be a group, $V$ a vector space over a ground field $F$, and $G$ acts on $V$ by linear transformations.
The typical problem in the invariant theory of the group $G$ is to find an upper bound for degrees of generators of $F[V]^G$. 
For fields $F$ of characteristic zero, there is a classical result of Noether \cite{noeth} which states that the algebra of invariants of $G$ 
is generated by polynomials of degrees not exceeding the order of $G$.

In this paper we are investigating a different problem and replace a generating set of invariants of $G$ by a single nonconstant invariant of $G$.
Namely, we are interested in a question: whether there is a nonconstant invariant of $G$ of degree not exceeding a certain value.

This question is motivated by security consideration in \cite{mzj} related to a public-key cryptosystem based on invariants of diagonalizable groups.  
Since one possible atttack on this cryptosystem is based on brute-force linear algebra,
if we know that there is a nonconstant invariant of $G$ of small degree, then this linear algebra attack is sucessful.
On this other hand, if we know that there are no nonconstant invariants of $G$ of small degree, then the cryptosystem is secure against this type of attack.

It is easier to formulate and investigate this problem in terms of the minimal degree $M_{G,V}$ of invariants of the group $G$ with respect to the fixed representation 
$G\to GL(V)$. We will establish both lower and upper bounds for $M_{G,V}$.

We start by recalling the concept of an invariant of a group $G$ in Section 1. In Section 2 we describe the public-key cryptosystem based on invariants of $G$. In Section 3 we show that the minimal degree of an abelian group $G$ is the same as the minimal degree of its subgroup generated by semisimple elements. 
We also study minimal degrees of diagonalizable groups. 
In Section 4 we relate the minimal degree $M_{G,V}$ of an algebraic $G$ to the minimal degrees of invariants of its torus $T$. 
Afterward, we explain the concept of invariants of supergroups in Section 5.  In Section 6 we derive certain properties of invariants of certain supergroups. One interesting property is that, unlike for groups, the basis of invariants for supergroups does not consist of monomials.
 
\section{Invariants of finitely-generated linear groups}

In this paper, we will consider only finitely generated groups $G$ acting faithfully on a finite-dimensional vector space $V=F^n$ over a field $F$ of arbitrary characteristics.
Therefore, we can asume that $G\subset GL(V)$. From the very beginning, assume that the representation $\rho:G\to GL(V)$ is fixed, and the group $G$ is given by a finite set of generators. With respect to the standard basis of $V$, each element $g$ of $G$ is therefore represented by an invertible matrix of size $n\times n$, and $g$ acts on vectors in $V$ by matrix multiplication. 

Let $F[V]=F[x_1, \ldots, x_n]$ be the algebra of polynomial functions on $GL(V)$. Then $G$ acts on $F[V]$ via 
$gf(v)=f(g^{-1}v)$, where $g\in G$, $f\in F[V]$ and $v\in V$.
An invariant $f$ of $G$ is a polynomial $f\in F[V]$, 
which has a property that its values are the same on orbits of the group $G$. In other words, for every vector $v\in V$ and for every element $g\in G$, we have $f(gv)=f(v)$.
We note that different representations of $G$ lead to different invariants in general, but this is not going to be a problem for us since our representation of $G$ is fixed.
We will denote the algebra of invariants of $G$ by $F[V]^G$.

Denote by $M_{G,V}$, or simply by $M_G$ or $M$ if we need not emphasise the group $G$ or the vector space $V$ it is acting on
the minimal positive degree of an invariant from $F[V]^G$. That is
$M_{G, V}=\min\{d>0 | F[V]_d^G\neq 0\}$. If $F[V]^G=F$, then we set $M_{G,V}=\infty$.

\section{Public key-cryptosystem based on invariants}\label{system}

We start by recalling the original idea of the public-key cryptosystem based on invariants from the paper \cite{dima1} and recalling its modification presented in \cite{dima2}.

\subsection{Cryptosystems based on invariants}

To design a cryptosystem, Alice needs to choose a finitely generated subgroup $G$ of $GL(V)$ for some vector space $V=F^n$ and a set $\{g_1, \ldots, g_t\}$ of generators of $G$. 
Alice also chooses an $n\times n$ matrix $a$. Alice needs to know a polynomial invariant $f: v\mapsto f(v)$ of this representation of $G$. Then the polynomial $af:v\mapsto f(av)$ 
is an invariant of the conjugate group $H=a^{-1}Ga$.

Depending on the choice $f$ and $a$, Alice chooses a set $M=\{v_0, \ldots, v_{s-1}\}$ of messages consisting of vectors from $V$ that are separated by the polynomial $af$.
This means that $f(av_i)\neq f(av_j)$ whenever $i\neq j$.

Alice also chooses a set of randomly generated elements $g_1, \ldots, g_m$ of $G$ (say, by multiplying some of the given generators of $G$), which generates a subgroup of $G$ that will be denoted by $G'$. 

Alice announces as a public key the set $M$ of possible messages, and the group $H=a^{-1}G'a$, conjugated to $G'$, by announcing its generators $h_i=a^{-1}g_ia$ for $i=1, \ldots, m$.

In the first paper \cite{dima1} its author assumes that the group $G$, its representation in $GL(V)$ and the invariant $f$ are in the public key. We refer to this setup as {\it variant one}. However, the version in paper \cite{dima2} assumes that $G$, its representation in $GL(V)$ and the invariant $f$ are secret.  We refer to this setup as {\it variant two}. We will comment on both variants later.

For the encryption, every time Bob wants to transmit a message $m\in M$, he chooses a randomly generated element $h$ of the group $H$(by multiplying some of the generators of $H$ given as a public key). 
Then he computes $u=hv_i$ and transmits the vector $u\in V$ to Alice.

To decript the message, Alice first computes $au$ and then applies the invariant $f$. (Of course this is the same as an application of the invariant $af$ of $H$ that separates elements of $M$). If $u=hv_i$, then $f(au)=f(ahv_i)=f(aa^{-1}gav_i)=f(gav_i)=f(av_i)$. Since $a$ was chosen so that
$f(av_i)\neq f(av_j)$ whenever $i\neq j$, Alice can determine from the value of $f(au)$ whether the symbol $v_i$ and the corresponding message that was encrypted by Bob.

\subsection{Design and modification of the cryptosystem based on invariants}

There is an obvious modification of the above cryptosystem which improves the ratio of the expansion in size from plaintext to ciphertext, namely replacing the set of two elements $v_0$ and $v_1$ from $V$ by a larger set $S=\{v_0, \ldots, v_{r-1}\}$, such that the invariant $f$ separates every two elements of $aS=\{av_0, \ldots, av_{r-1}\}$ instead. 

The paper \cite{mzj} studies cryptosystems based on invariants of finitely generated groups $G$ and considers advantages and disadvantages of various choices of $G$. Most notable is the distinction between diagonalizable and unipotent groups as well between finite and infinite groups. The behaviour of the cryptosystem varies based on the choice of the underlying ground field $F$ or residue ring $R$. 
When working over finite field, the cyclicity of the multiplicative group $F^{\times}$ plays a big role and security of the cryptosystem is related to the discrete logarithm problem. 
When $F$ is a number field, then the factorization properties in the ring of its integers $Z$ come into forefront. Finally, in the case of a residue ring $R$ of a ring of algebraic integers $Z$ modulo its ideal $\mathfrak{a}$, we work over a group of units of a finite ring and their multiplicative structure is more involved than that for a finite field. This case also involves questions related to factorization in the ring of algebraic integers $Z$ and is therefore a mixture between the previous two cases.

\subsection{Linear algebra attack on the cryptosystem}\label{LA}

The notion of the minimal positive degree of an invariant and the value of $M=M_{G,V}$ are important for the security of the invariant-based cryptosystem (both variants one and two) we are considering.
For example, if we know that $M_G$ is so small that $m\binom{n+M-1}{M}=O(n^r)$ is polynomial in $n$, then Charlie can find an invariant $f'$ of $G$
in polynomial time by solving consecutive linear systems for $d=1, \ldots, \binom{n+M-1}{M}$, each consisting of $m\binom{n+d-1}{d}$ equations in the $\binom{n+d-1}{d}$ variables described in the previous section. For a fixed $d$, this can be accomplished in time $O(m(\binom{n+d-1}{d})^4)$ and the total search will take no more than time
$O(n^{8r})$.
Therefore, for the security of the system it must be guaranteed that $m\binom{n+M-1}{M}$ is high, say, it is not polynomial in $n$.

\section{Lower bounds for degrees of polynomial invariants}

The significance of understanding the minimal degree $M_{G,V}$ of invariants for the security of the invariant-based cryptosystem was established above. In particular, it is important to find a nontrivial lower bound for $M_{G,V}$.  Unfortunately, we are not aware of any articles establishing lower bounds for the minimal degree of invariants, except in very special circumstances, e.g. \cite{huf}.

On the other hand, there are numerous upper bounds for the minimal degree $\beta(G,V)$ such that $F[V]^G$ is generated as an algebra by all invariants in degrees 
not exceeding $\beta(G,V)$. 
For example, a classical result of Noether \cite{noeth} states that if the characteristic of $F$ is zero and $G$ is finite of order $|G|$, then $\beta(G,V)\leq |G|$. 
There is an extensive discussion of Noether bound and results about $\beta(G,V)$ in section 3 of \cite{smith}.
It was conjectured by Kemper that for $G\neq 1$, and arbitratry ground field $F$, the number $\beta(G, V)$ is at most $\dim V(|G|-1)$. Recently, this conjecture was proved by Symonds in \cite{sym}.

When one wants to find an invariant of $G$, it seems natural to consider an upper bound $\beta(G,V)$. However, if we want to show that there are no invariants of small degrees 
(as is our case), then we need to find lower bounds for $M_{G,V}$. Until now, there was no real impetus to consider such a problem.

Assume again that $G$ is a (finitely generated) subgroup of $GL(V)$, and denote $M_{G, V}$ just by $M_G$.

Denote by $\mathbb{G}=\overline{G}$ the Zariski closure of $G$. We will assume that $\mathbb{G}$ is a linearly reductive subgroup in $GL(V)$
(in particular, this assumption is satisfied if $G$ is a finite group and the characteristic of $F$ does not divide $|G|$). 
According to \cite{hochrob} (see also \cite{derkraft}),  $F[V]^G=F[V]^{\mathbb{G}}$ is a Cohen-Macaulay algebra. Therefore
$F[V]^G$ is a free module over its subalgebra $F[p_1, \ldots , p_s]$, freely generated by the (homogeneous) parameters $p_1, \ldots , p_s$, which are called the first generators.
In other words, $F[V]^G=\oplus_{1\leq i\leq l}F[p_1, \ldots , p_s]h_i$, where $h_1 , \ldots, h_l$ are called the second generators.
If $F[V]^G\neq F$, then $M_G=\min\{\{\deg h_i > 0\},\{\deg p_j\}\}$.

In what follows we will denote by $\zeta_k$ a primitive root of unity of order $k$. If the order $k$ is clear from the context, we will denote it just by $\zeta$. 
Additionally, every time $\zeta_k$ is mentioned, we assume that it is an element of the ground field $F$.

If a matrix $g\in GL(V)$ has a finite order $k$, then all eigenvalues $\lambda_1, \ldots , \lambda_n$ of $g$ are roots of unity.
If we denote $\zeta=\zeta_k$, then there are integers $k_i$ such that $\lambda_i=\zeta^{k_i}$, where $0\leq k_i < k$ and $gcd(k_1, \ldots , k_n, k)=1$.
For $g\neq 1$ denote by $k_g$ the positive integer 
\[k_g=\min\{\sum_{i=1}^n a_i >0 | \sum_{i=1}^n a_i k_i \equiv 0\!\!\!\!\pmod k, \mbox{ where integers } a_1, \ldots , a_n \geq 0\}.\]

The following lemma describes invariant polynomials and $M_{\langle t \rangle}$ for a diagonal matrix $t$ of finite order.

\begin{lm}\label{suchandsuch}
Assume that $t$ is a diagonal matrix of the finite order $k$ with diagonal entries $\lambda_1=\zeta_k^{k_1}, \ldots, \lambda_n=\zeta_k^{k_n}$, where the exponents $k_i$ are as above. 
Then the invariant subalgebra of $F[V]^{\langle t \rangle}$ is generated by monomials $x^a=x_1^{a_1}\ldots x_n^{a_n}$ such that $\sum_{i=1}^n a_ik_i \equiv 0 \pmod k$.
Additionally, if $t\neq 1$, then $M_{\langle t \rangle}=k_t$.
\end{lm}
\begin{proof}
The properties of numbers $k_i$ follow immediately. 
Since $t$ acts on the corresponding coordinate function as $tx_i=\lambda_i^{-1}x_i$, we obtain that a monomial $x^a=x_1^{a_1}\ldots x_n^{a_n}$ 
is a invariant of $F[V]$ if and only if $\sum_{i=1}^n a_ik_i \equiv 0 \pmod k$.
Because every monomial $x^b$ is a semi-invariant of $t$, monomials $x^a$ as above generate $F[V]^{\langle t \rangle}$. The formula for $M_{\langle t \rangle}$ is then clear.
\end{proof}

For the next lemma we apply standard results from algebraic group theory, that can be found, for example, in \cite{hum, water}. Assume that $F$ is a perfect field.
For an element $g\in G$ let $g=g_s g_u$ be its Jordan-Chevalley decomposition. Let $G_s$ and $G_u$ denote the sets of semisimple and unipotent components of all elements from $G$, respectively.
\begin{lm}\label{the caseofcyclicgroup}
Assume that the ground field $F$ is perfect. If a group $H$ is abelian, then $M_H=M_{<H_s>}$. 
In particular, if $H$ is an abelian subgroup of $G$, then $M_{<H_s>}\leq M_G$.
\end{lm}
\begin{proof}
Since the algebraic group $\mathbb{H}=\overline{H}$ is abelian, it can be written as a product 
$\mathbb{H}=\mathbb{H}_s\times \mathbb{H}_u$ of its closed subgroups $\mathbb{H}_s$ and $\mathbb{H}_u$. The inclusions $H_s\subseteq \mathbb{H}_s$ and $H_u\subseteq \mathbb{H}_u$
imply that $\mathbb{H}_s=\overline{<H_s>}$ and $\mathbb{H}_u=\overline{<H_u>}$.

Furthermore, $F[V]^H=F[V]^{\mathbb{H}}=(F[V]^{\overline{<H_s>}})^{\overline{<H_u>}}$. Since the group $\overline{<H_u>}$ is unipotent,
$F[V]_d^{\overline{<H_s>}}\neq 0$ implies $F[V]_d^{\mathbb{H}}=(F[V]_d^{\overline{<H_s>}})^{\overline{<H_u>}}\neq 0$. This means that $M_H=M_{\overline{<H_s>}}=M_{<H_s>}$. 

Since $H\leq G$ implies $M_H\leq M_G$, the second statement follows.
\end{proof}

A subgroup $G$ of $GL(V)$ is called {\it small}, if there is an abelian subgroup $H$ of $G$ such that $M_G =M_H$.

\begin{lm}\label{aboundforoneelement}
Assume that the ground field $F$ is perfect.
If $g\neq 1$ is of finite order, then $M_{<g>}=k_g$. In particular, if $G$ is finite, then $\max\{k_g ; g\in G, g\neq 1\}\leq M_G$.
\end{lm} 
\begin{proof}
Lemma \ref{the caseofcyclicgroup} implies $M_{<g>}=M_{<g_s>}$. With respect to a basis of $V$, consisting of eigenvectors of $g_s$, $g_s$ is represented by a diagonal matrix.
By Lemma \ref{suchandsuch} we obtain $M_{<g_s>}=k_{g_s}$. Since $k_{<g>}=k_{<g_s>}$, the lemma follows.
\end{proof}
The following lemma is well-known, see \cite{burn}.

\begin{lm}\label{burn}
If $G\subset GL_n(\mathbb{R})$ and $G$ is finite, then $G$ has an invariant of degree two.
\end{lm}
\begin{proof}
Let $g_1=1, \ldots, g_s$ be all elements of $G$ and $\mathbb{R}[V]=\mathbb{R}[t_1, \ldots, t_n]$. Denote by $x_i=g_i(t_1^2+\ldots +t_n^2)$ for $i=1, \ldots, s$.  Since values of 
each $x_i$ are non-negative when evaluated as polynomials in $t_1, \ldots, t_n$, the values of the invariant polynomial $\sum_{i=1}^s x_i$ 
evaluated as polynomial in $t_1, \ldots, t_n$ are non-negative and they can be equal to zero only if each 
$x_i$ is zero. But $x_1=0$ only if $t_1=\ldots =t_n=0$. Therefore $\sum_{i=1}^s x_i$ is positive definite quadratic form in $t_1, \ldots, t_n$, 
hence a non-zero invariant of $G$.
\end{proof}

Lemma \ref{aboundforoneelement} and Lemma \ref{burn} have the following interesting consequence.
\begin{cor}\label{F=R}
Let $g\neq 1$ correspond to a matrix from $GL_n(\mathbb{R})$ of finite order. Then either one of the eigenvalues of $g$ equals $1$ or there are two
eigenvalues $\lambda$ and $\mu$ of $g$, both different from $1$ such that $\lambda\mu =1$.  
\end{cor}
\begin{lm}
Let $G$ be a finite abelian group of an exponent $q$, the ground field $F$ is perfect and $char F$ does not divide $q$. 
Then for every $G$-module $V$ one has the upper bound $M_{G, V}\leq q$. This upper bound is sharp.
\end{lm}
\begin{proof}
Without a loss of generality one can assume that $G\leq GL(V)$. By Lemma \ref{the caseofcyclicgroup}
one can also assume that $G=G_s$, hence $G$ is diagonalizable. Every element $g\in G$ is represented by a matrix whose diagonal entries are powers of the $q$-th primitive root $\zeta$. This implies the first statement. To show that the upper bound is sharp, it is enough to consider an example when one element $g$ is represented by a matrix whose all diagonal entries are equal to $\zeta$.
\end{proof}
If $G$ is a diagonalizable finite abelian subgroup of $GL(V)$, then using Lemma 3.1 of \cite{hl} we can reduce the computation of $M_{G, V}$ to an integer programming problem. 
In fact, this lemma states that there are invariant monomials 
\[f_1=t_1^{m_1}, f_2=t_1^{v_{12}}t_2^{m_2}, \ldots, f_n=t_1^{v_{1 n}}\ldots t_{n-1, n}^{v_{n-1, n}}t_n^{m_n}\]
of a "triangular shape", where $m_n>0$ and $m_i > v_{i j}\geq 0$ for $1\leq i < j\leq n$, 
such that every invariant monomial from the field of rational invariants $F(V)^G$ is a product of (not necessary non-negative) powers of the monomials $f_1, \ldots, f_n$. 
Since $F[V]^G$ has a basis consisting of invariant monomials, any such monomial has a form  $f_1^{l_1}\ldots f_n^{l_n}$, where $l=(l_1, \ldots , l_n)\in\mathbb{Z}^n$ is a solution of the system of inequalities
\[m_k l_k +\sum_{j> k} v_{k j}l_j\geq 0 \text{ for } 1\leq k\leq n.\]
From here we derive that $M_{G, V}$ is the minimum of the function 
\[\sum_{1\leq i\leq n}(m_i+\sum_{j< i}v_{j i})l_i\]
evaluated on the solution set of the above system of inequalities.   

To illustrate the difficulty of finding a lower bound for $M_{G,V}$, we will determine the value of $M_{G,V}$ explicitly for certain finite subgroups 
$G$ of $GL_2(\mathbb{C})$. 
The list of all finite subgroups of $GL_2(\mathbb{C})$ is presented in \cite{huf}.

Let $G$ be a finite group from Lemma 2.1 of \cite{huf}. 
The group $G$ has two generators
\[A=\left(\begin{array}{cc}
\lambda^{v_1} & 0 \\
0 & \lambda^{j v_2}
\end{array}\right) , \ B=\left(\begin{array}{cc}
\lambda^{g} & 0 \\
0 & \lambda^{dg}
\end{array}\right),\]
where $\lambda$ is an $e$-th primitive root of unity, $v_1, v_2 > 1, v_1 v_2|g, g|e, d|e, gcd(v_1, v_2)=gcd(e, j)=gcd(v_1, d)=gcd(v_2, d)=1$.
Additionally, the number $d$ is square-free and each prime factor of $e$ divides one of the numbers $v_1$, $v_2$ or $d$.
In particular, $G\simeq <A>\times <B>=\mathbb{Z}_{e}\times\mathbb{Z}_{\frac{e}{g}}$.

To calculate $M_G$, we need to consider the following system of congruencies:
\[v_1 a_1+jv_2 a_2\equiv 0\!\!\pmod e,  \ ga_1 +dga_2 \equiv 0\!\!\pmod e, \]
where $a_1, a_2\geq 0$ are such that $a_1+a_2> 0$.
The second congruence implies that $a_1=\frac{et}{g} -da_2$, where $t$ is a positive integer. Substituting the value of $a_1$ into the first congruence we receive
\[\frac{ev_1 t}{g}=(dv_1-jv_2)a_2 \pmod e.\]
Since $gcd(e, dv_1-jv_2)=1$, we obtain that $\frac{ev_1}{g}$ divides $a_2$, which implies that $\frac{ev_2}{g}$ divides $a_1$. Since both $a_1$ and $a_2$ are multiples of 
$\frac{e}{g}$, the second congruence $ga_1 +dga_2 \equiv 0\!\!\pmod e$ can be eliminated from the system since it is automatically satisfied.

Define $a_1=\frac{ev_2}{g} a'_1, a_2=\frac{ev_1}{g}a'_2$. Then
$a'_1 +ja'_2 =0\pmod{\frac{g}{v_1 v_2}}$, or equivalently, $a'_1 +ja'_2 =\frac{gs}{v_1 v_2}$ for some $s> 0$. This congruence has the solution
\[a'_1=\frac{gs(j+1)}{v_1 v_2} -jt, \ a'_2 = -\frac{gs}{v_1 v_2} +t .\]
Since $a'_1, a'_2\geq 0$, the parameter $t$ satisfies
\[\frac{gs}{v_1 v_2}\leq t\leq \frac{gs}{v_1 v_2} +[\frac{gs}{jv_1 v_2}].\]
Additionally,
\[a_1=\frac{es(j+1)}{v_1}-\frac{ev_2 jt}{g} \text{ and } \ a_2= -\frac{es}{v_2} +\frac{ev_1 t}{g}.
\]
Thus
\[a_1+a_2 =\frac{es(j+1)}{v_1}-\frac{es}{v_2}-\frac{et}{g}(v_2 j-v_1).
\]
Finally, observe that for every $s>0$ and for every $t$ such that $\frac{gs}{v_1 v_2}\leq t\leq \frac{gs}{v_1 v_2} +[\frac{gs}{jv_1 v_2}]$, 
the right-hand-side of the above formula for $a_1+a_2$ is greater than zero.

Now are are ready to determine the values of $M_{G,V}$.
\begin{pr}\label{aformula}
Assume $G$ is a finite group from Lemma 2.1 of \cite{huf}, as above. Then the value of $M_{G,v}$ is given as follows.
If $jv_2 < v_1$, then $M_G=\frac{e}{v_1}$.
If $jv_2 > v_1$, then $M_G=\min\{\min\limits_{0< s < j}\{\frac{s}{v_1}-[\frac{gs}{jv_1 v_2}]\frac{e(v_2 j -v_1)}{g}\}, \frac{e}{v_2}\}$.
\end{pr}
\begin{proof}
If $jv_2 < v_1$, and $s$ is fixed, then the minimum of such $a_1+a_2$ equals $\frac{es}{v_1}$ and is attained for $t=\frac{gs}{v_1 v_2}$. Therefore
$M_G=\min\{a_1 +a_2\}=\frac{e}{v_1}$.

If $jv_2 > v_1$, and $s$ is fixed, then the minimum of such $a_1+a_2$ equals $\frac{es}{v_1}-[\frac{gs}{jv_1 v_2}]\frac{e(v_2 j -v_1)}{g}$ and is attained for 
$t=\frac{gs}{v_1 v_2}+[\frac{gs}{jv_1 v_2}]$.

If $s=jl+s'$, where $0\leq s' < j$, then $[\frac{gs}{jv_1 v_2}]=\frac{gl}{v_1 v_2}+[\frac{gs'}{jv_1 v_2}]$. After substituting this into the above expression for $a_1+a_2$ 
we obtain
\[
a_1+a_2=\frac{el}{v_2} +(\frac{s'}{v_1}-[\frac{gs'}{jv_1 v_2}]\frac{e(v_2 j -v_1)}{g}).
\]
If $s'=0$, then the minimum for such $a_1+a_2$ is attained for $l=1$ and it equals to $a_1+a_2=\frac{e}{v_2}$. 
If $s' > 0$, then the minimum for such $a_1+a_2$ is attained for $l=0$ and it equals to 
\[\min\limits_{0<s'<j}\{\frac{s'}{v_1}-[\frac{gs'}{jv_1 v_2}]\frac{e(v_2 j -v_1)}{g}\}.\]
The statement follows by combination of the last two formulas.
\end{proof}
\begin{ex}
The following example shows that not all finite subgroups of $GL(V)$ are small.
Let $G$ be a subgroup of $SL_2(\mathbb{C})$ generated by the matrices
\[
\left(\begin{array}{cc}
-1 & 0 \\
0 & -1
\end{array}\right), 
\frac{1}{2}\left(\begin{array}{cc}
-1+i & 1-i \\
-1-i & -1-i
\end{array}\right), 
\left(\begin{array}{cc}
0 & 1 \\
-1 & 0
\end{array}\right),
\left(\begin{array}{cc}
i & 0 \\
0 & -i
\end{array}\right).
\]
The group $G$ is the group from Lemma 2.3 of \cite{huf} and $V=\mathbb{C}^2$. If $g\in G$ is not an identity matrix, then it has eigenvalues $\lambda$ and $\lambda^{-1}$, where $\lambda\neq 1$ is a root of unity. If $H$ is an abelian subgroup of $G$, then $H_s$ can be conjugated with a subgroup $H'$ of the group of diagonal matrices. 
Thus $x_1 x_2\in \mathbb{C}[V]^{H'}$, 
i.e. $M_H=M_{H_s}\leq 2$. On the other hand, Lemma 4.1 of \cite{huf} (see the first row in the table on page 327) implies $M_G=6$.
\end{ex}

Based on the above discussion, the following problem seems natural.
\begin{quest}\label{question}
Characterize the class of small finite subgroups $G$ of $GL(V)$.
\end{quest}

A more general problem is to estimate the value of $M_G$ for a given finite subgroup $G\leq GL(V)$. There are no general results for the lower bound for $M_G$ but 
the following result of Thompson gives an upper bound for $M_G$ in general. 
\begin{pr}\label{thompson'sresult}
If $G$ is a finite subgroup of $GL_n(\mathbb{C})$ and $G$ has no non-trivial characters, then $M_G\leq 4n^2$.
\end{pr}
\begin{proof}
In the notation of the paper \cite{thom}, the integer $M_G$ coincides with $d_G$. The main theorem of \cite{thom} states that $d_G\leq 4n^2$.
\end{proof} 

\section{Minimal degrees of invariants of algebraic groups}

Let $\mathbb{G}$ be an algebraic subgroup of $GL(V)$ and $\mathbb{B}$ be its Borel subgroup. Propositions I.3.4 and I.3.6 of \cite{jan} (see also Theorem 9.1 of \cite{gross}) imply
\[(F[V]\otimes F[\mathbb{G}/\mathbb{B}])^{\mathbb{G}}\simeq F[V]^{\mathbb{B}}.\] 
Since $\mathbb{G}/\mathbb{B}$ is a projective variety, we have $F[\mathbb{G}/\mathbb{B}]=F$. Therefore, $F[V]^{\mathbb{G}}\simeq F[V]^{\mathbb{B}}$ and 
the minimal degrees of invariants $M_{\mathbb{G}, V}$ and $M_{\mathbb{B}, V}$ coincide.

The group $\mathbb B$ is a semi-direct product of a torus $\mathbb T$ and the unipotent radical $\mathbb U$ of $\mathbb B$, i.e. $\mathbb{B}=\mathbb{T}\ltimes\mathbb{U}$. 
For a (finite-dimensional) $\mathbb U$-module $S$, denote by $S_{\mathbb{U}}$ the smallest $\mathbb U$-submodule of $S$ such that $\mathbb U$ acts trivially on $S/S_{\mathbb{U}}$. 

Define a filtration of a $\mathbb U$-module $V$ as
\[0\subseteq V_k\subseteq V_{k-1}\subseteq\ldots\subseteq V_1\subseteq V_0=V,\]
where $V_{i+1}=(V_i)_{\mathbb U}$ for each $0\leq i\leq k$. Since $\mathbb{U}\unlhd\mathbb{B}$, the above filtration is also a filtration of $\mathbb B$-submodules. 
One can verify easily that $(V/V_1)^*=(V^*)^{\mathbb U}$, which implies that $F[V/V_1]$ is a $\mathbb B$-invariant subalgebra of $F[V]$ such that $F[V/V_1]\subseteq F[V]^{\mathbb U}$. 
\begin{pr}\label{boundsviatorus}
The minimal degrees of invariants of $G$ and $T$ are related in the following way.
\[M_{\mathbb{T}, V}\leq M_{\mathbb{G}, V}=M_{\mathbb{B}, V}\leq M_{\mathbb{T}, V/V_1}.\]
\end{pr}
 \begin{proof}
First inequality is trivial. For the second inequality, first observe that $F[V]^{\mathbb B}=(F[V]^{\mathbb U})^{\mathbb T}$. Therefore, 
$F[V/V_1]^{\mathbb T}\subseteq F[V]^{\mathbb B}$ which implies $M_{\mathbb{B}, V}\leq M_{\mathbb{T}, V/V_1}$. 
\end{proof}
The second inequality in the above proposition is sharp. In fact, if $\mathbb{U}$ coincides with the centralizer of the flag $V_k\subseteq V_{k-1}\subseteq\ldots\subseteq V_1\subseteq V$, then $\mathbb U$ is good in the sense of \cite{pom}. Furthermore, Theorem 4.2 of \cite{pom} implies that $F[V]^{\mathbb U}=F[V/V_1]$. Thus $F[V]^{\mathbb B}=F[V/V_1]^{\mathbb T}$, hence $M_{\mathbb{G}, V}=M_{\mathbb{B}, V}=M_{\mathbb{T}, V/V_1}=M_{\mathbb{T}, V}$.

An important consequence of the above proposition is that in many cases the minimal degree of invariants of a linear group is controlled by minimal degree of invariants of its suitable abelian subgroup; more precisely, by its diagonalizable subgroup. 

\section{Invariants of supergroups}

Having in mind possible modification of the cryptosystem based on invariants of groups to a cryptosystem based on supergroups, we will define the notion of an invariant of a supergroup.
From now on we assume that the characteristic of the ground field $F$ is different from $2$. 

\subsection{Definitions and actions}

Let $V$ be a superspace, that is a $\mathbb{Z}_2$-graded space with even and odd components $V_0$ and $V_1$,
respectively. If $v\in V_i$, then $i$ is said to be a {\it parity} of $v$ and it is denoted by $|v|$. In what follows, morphisms between two superspaces $V$ and $W$ are assumed to be graded. The tensor product $V\otimes W$ has the natural structure of a superspace given by
$(V\otimes W)_i=\bigoplus\limits_{k+l=i, k, l\in\mathbb{Z}_2} V_k\otimes W_l$. 

A $\mathbb{Z}_2$-graded associative algebra $A$ is called a {\it superalgebra}.
The superalgebra $A$ is said to be {\it supercommutative} if it satisfies $ab=(-1)^{|a||b|}ba$ for all homogeneous elements $a$ and $b$.
For example, any algebra $A$ has the trivial superalgebra structure defined by $A_0=A, A_1=0$. The tensor product $A\otimes B$ of two superalgebras $A$ and $B$
has the superalgebra structure defined by
\[(a\otimes b)(c\otimes d)=(-1)^{|b||c|}ac\otimes bd\]
for $a, c\in A$ and $b, d\in B$.
The category of all supercommutative superalgebras with graded morphisms is denoted by $\mathsf{SAlg}_F$.

A superalgebra $A$ is called a {\it superbialgebra} if it is a coalgebra with the coproduct $\Delta : A\to A\otimes A$ and counit
$\epsilon : A\to F$ such that both $\Delta$ and $\epsilon$ are superalgebra homomorphisms. In what follows we use Sweedler's notation
$\Delta(a)=\sum a_1\otimes a_2$ for $a\in A$. Let $A^+$ denote the (two-sided) superideal $\ker\epsilon$.

A superspace $V$ is called a left/right $A$-{\it supercomodule} if $V$ is a left/right $A$-comodule and the corresponding comodule map
$\tau : V\to V\otimes A$ is a morphism  of superspaces.

A superbialgebra $A$ is called a {\it Hopf superalgebra} if there is a superalgebra endomorphism $s : A\to A$ such that 
$\sum a_1s(a_2)=\sum s(a_1)a_2=\epsilon(a)$ for $a\in A$. Additionally, we assume that $s$ is bijective and it satisfies the condition 
$\Delta s=t(s\otimes s)\Delta$, where $t : A\otimes A\to A\otimes A$ is a (supersymmetry) homomorphism defined by 
$a\otimes a'\mapsto (-1)^{|a|||a'}a'\otimes a$ for $a, a'\in A$.

Let $A$ be a supercommutative superalgebra. Then the functor $SSp \ A : \mathsf{SAlg}_F\to \mathsf{Sets}$, defined by 
$SSp \ A (C)=Hom_{\mathsf{SAlg}_F}(A, C)$ for $C\in\mathsf{SAlg}_F$, is called an {\it affine superscheme}. If $X=SSp \ A$ is an affine superscheme, then
$A$ is denoted by $F[X]$ and it is called the {\it coordinate superalgebra} of $X$.

If $A$ is a Hopf superalgebra, then $G=SSp \ A$ is a group functor that is called an {\it affine group superscheme}, or shortly, an {\it affine supergroup}. 
The group structure of $G(C)$ is given by $g_1 g_2 (a)=\sum g_1(a_1)g_2(a_2)$, $g^{-1}=g s$ and $1_{G(C)}=\epsilon$ for $g_1, g_2, g\in G(C)$ and $a\in A$. 
The category of affine supergroups is dual to the category of supercommutative Hopf superalgebras. 
If $F[G]$ is finitely generated, then $G$ is called an {\it algebraic} supergroup. If $F[G]$ is finite-dimensional, then $G$ is called a {\it finite} supergroup. 

A (closed) subsupergroup $H$ of $G$ is uniquely defined by the Hopf ideal $I_H$ of $F[G]$ such that for every $C\in\mathsf{SAlg}_F$ an element $g\in G(C)$ belongs to $H(C)$ if and only if $g(I_H)=0$. For example, the {\it largest even} subsupergroup $G_{ev}$ of $G$ is defined by the ideal $F[G]F[G]_1$.

The category of left finite-dimensional $G$-supermodules coincides with the category of right $F[G]$-supercomodules. In fact, if $V$ is a right $F[G]$-supercomodule, then $G(C)$ acts on $V\otimes C$ by $C$-linear transformation
$g(v\otimes 1)=\sum v_1\otimes g(a_2)$ for $g\in G(C)$ and $\tau(v)=\sum v_1\otimes a_2$.

Let $V$ be a superspace such that $\dim V_0 =m$ and $\dim V_1 =n$. The superspace $V$ corresponds to an affine superscheme $A^{m|n}$, called the {\it affine superspace} of (super)dimension $m|n$, such that $A^{m|n}(C)=C_0^m\oplus C_1^n$ for every $C\in\mathsf{SAlg}_F$. The affine superscheme $A^{m|n}$ can be identified with the functor $(V\otimes ?)_0$. In fact, choose a homogeneous basis consisting of elements $v_i$ such that $|v_i|=0$ for $1\leq i\leq m$ and $|v_i|=1$ for $m+1\leq i\leq m+n$. Then every element $w$ of $(V\otimes C)_0$ has the form $w=\sum_{1\leq i\leq m+n} v_i\otimes c_i$, where $|c_i|=|v_i|$. 

The coordinate superalgebra of $A^{m|n}$ is isomorphic to the polynomial superalgebra freely generated by the dual basis $x_i$ of $V^*$ such that 
$x_i(v_j)=\delta_{ij}$ for $1\leq i, j\leq m+n$. In other words, $w(x_i)=x_i(w)=c_i$ for every $w=\sum_{1\leq i\leq m+n} v_i\otimes c_i\in (V\otimes C)_0$ and $C\in\mathsf{SAlg}_F$. 
In order to make the notation consistent, we will also denote $F[A^{m|n}]$ by $F[V]$.

Every $g\in G(C)$ induces an even operator on the $F$-superspace $V\otimes C$. Thus $(V\otimes C)_0$ is a $G(C)$-submodule of $V\otimes C$. Since this action is functorial, it gives
the left $G$-action on the affine superscheme $A^{m|n}$.
The composition of this action with the inverse morphism $g\mapsto g^{-1}$ defines the right action of $G$ on $A^{m|n}$, which is equivalent to the right coaction of $F[G]$ on $F[V]$. 

Since the comodule map  $F[V]\to F[V]\otimes F[G]$ is a superalgebra homomorphism, the $F[G]$-supercomodule structure of $F[V]$ is defined by $F[G]$-supercomodule structure of $V^*=\sum_{1\leq i\leq m+n} Fx_i$.
If $\tau(v_i)=\sum_{1\leq k\leq t} v_k\otimes a_{ki}$ for $1\leq i\leq m+n$, then 
\[\tau(x_i)=\sum_{1\leq k\leq t} x_k \otimes (-1)^{|v_k|(|v_i|+|v_k|)}s(a_{ik}).\] 

There is a natural pairing $(F[V]\otimes C)\times (V\otimes C)\to C$ given by 
\[(f\otimes a)(v\otimes b)=(-1)^{|a||v|} f(v)ab=(-1)^{|a||v|}v(f)ab\] 
for $a, b\in C$ and $C\in\mathsf{SAlg}_F$, such that the above coaction is equivalent to the standard action 
$(g(f\otimes a))(v\otimes b)=(f\otimes a)(g^{-1}(v\otimes b))$ for $g\in G(C)$.

\subsection{Cryptology application}

The invariants of supergroups have two possible applications in the design of public-key cryptosystem.
The first option is to work with {\it relative} invariants from the $C$-superalgebra $C[V]^{G(C)}=(F[V]\otimes C)^{G(C)}$ for some superalgebra $C\in A\in\mathsf{SAlg}_F$.
The second option is to work with {\it absolute} invariants from the superalgebra $F[V]^G$, consisting of all $f\in F[V]$ such that $\tau(f)=f\otimes 1$, or equivalently, 
$g(f\otimes 1)=f\otimes 1$ for every $g\in G(C)$ and $C\in\mathsf{SAlg}_F$.
We will leave a consideration of these options for the future.

\section{Invariants of certain supergroups}

We will now investigate the structure of invariants of certain supergroups $G$. We will establish, in contrast to the case of diagonalizable groups, that generators of invariants of 
$G$ are not given by monomials.

Recall that every diagonalizable algebraic group is isomorphic to a finite product of copies of the one-dimensional torus $G_m$ and groups $\mu_n$, where
$\mu_n$ is the $n$-th roots of unity and $n>1$. Here $\mu_n(C)=\{c\in C^{\times} | c^n=1\}$ for every commutative algebra $C$ (see Theorem 2.2 of \cite{water}).

Let $D$ be a diagonalizable algebraic group and $X=X(D)$ be the character group of $D$. Then $F[D]=FX$ is a group algebra of $X$. 
The Lie algebra $Lie(D)$ can be identified with the subspace of $F[D]^*=(FX)^*$ consisting of all linear maps 
$y : FX\to F$ such that $y(g_1 g_2)=y(g_1)+y(g_2)$ for every $g_1, g_2\in X$. 
Fix a pair $(g,x)$, where $g\in X$ and $x\in Lie(D)$ such that if $x\neq 0$ then $g^2=1$. 
Since $char F\neq 2$, we have $y(g)=0$ for every $y\in Lie(D)$.

The following supergroup $D_{g, x}$ was first introduced in \cite{maszub}. 
The coordinate algebra $F[D_{g, x}]$ is isomorphic to $FX\otimes F[z]=FX\oplus (F X)z$, where $z$ is odd and $z^2=0$. 
The Hopf superalgebra structure on $F[D_{g, x}]$ is defined as:
\[\Delta(h)=h\otimes h + x(h)hz\otimes hgz, \ \Delta(z)=1\otimes z+z\otimes g, \ \epsilon(z)=0, \ \epsilon(h)=1,\]
$s(h)=h^{-1}$ for $h\in X$ and $s(z)=-g^{-1}z$. 

Denote $Fh\oplus Fhz$ by $L(h)$. Then every $L(h)$ is an indecomposable injective $D_{g, x}$-supersubmodule of $F[D_{g, x}]$ and  $F[D_{g, x}]=\oplus_{h\in X} L(h)$.
Let $Y$ denote $\{h\in X| x(h)=0\}$. 
The supermodule $L(h)$ is irreducible if and only if $h\not\in Y$. If $L(h)$ is not irreducible, then it has the socle $S(h)=Fh$ and $L(h)/S(h)\simeq \Pi S(gh)$.

If we denote the basis elements $h$ and $hz$ of $L(h)$ by $f_0$ and $f_1$ respectively, then
\[\tau(f_0)=f_0\otimes h + x(h)f_1\otimes hgz \text{ and } \tau(f_1)=f_0\otimes hz +f_1\otimes hg.\]
Also, $L(h)^*\simeq \Pi L(g^{-1}h^{-1})$ and $S(h)^*\simeq S(h^{-1})$.
\begin{pr}\label{irred}(Proposition 5.1 of \cite{maszub})
Every irreducible $D_{g, x}$-supermodule is isomorphic either to $L(h)$ for $h\not\in Y$ or to $S(h)$ for $h\in Y$. Moreover, every finite-dimensional $D_{g, x}$-supermodule
is isomorphic to a direct sum of (not necessary irreducible) supermodules $\Pi^a L(h)$ and $\Pi^b S(h')$ for $h\in X, h'\in Y$ and $a, b=0, 1$.
\end{pr}

Consider a (finite-dimensional) $D_{g, x}$-supermodule $V$
such that $V^*\simeq V(h_1)\oplus\ldots\oplus V(h_s)$.
The superalgebra $F[V]$ is generated by the elements
$f_{j, 0}$ and $f_{j, 1}$, for $1\leq j\leq s$, such that $|f_{j, 0}|=0, |f_{j, 1}|=1$ and
\[\tau(f_{j, 0})=f_{j, 0}\otimes h_j +x(h_j)f_{j, 1}\otimes h_j gz \text{ and } \tau(f_{j, 1})=f_{j, 0}\otimes h_j z+ f_{j, 1}\otimes h_j g.\]

Let $l=(l_1, \ldots, l_s)$ be a vector with non-negative integer coordinates and let $J$ be a subset of $\underline{s}=\{1, 2, \ldots, s\}$.
Denote $f_0^l=\prod_{1\leq j\leq s}f_{j, 0}^{l_j}$, $f_1^J=\prod_{j\in J}f_{j, 1}$, $h^l=\prod_{1\leq j\leq s} h_j^{l_j}$ and $h^J=\prod_{j\in J}h_j$. 
For $1\leq j\leq s$ let $\epsilon_j$ denote the vector that has the $j$-th coordinate equal to $1$ and all remaining coordinates equal to zero.

For a basis monomial $f_0^l f_1^J$ we have 
\[\begin{aligned}\tau(f_0^l f_1^J)=&
(f_0^l\otimes h^l+\sum_{1\leq j\leq s}l_j x(h_j)f_0^{l-\epsilon_j}f_{j, 1}\otimes h^l gz)\times \\
&(f_1^J\otimes h^J g^{|J|}+\sum_{j\in J} (-1)^{k_{j, J}}f_{j, 0}f_1^{J\setminus j}\otimes h^J g^{|J|-1}z) \\
=&f_0^l f_1^J\otimes h^l h^J g^{|J|}+\sum_{j\not\in J}(-1)^{k_{j, J\cup j}}l_j x(h_j)f_0^{l-\epsilon_j}f_1^{J\cup j} \otimes h^l h^J g^{|J|+1}z \\
&+\sum_{j\in J}(-1)^{k_{j, J}} f_0^{l+\epsilon_j}f_1^{J\setminus j}\otimes h^l h^J g^{|J|-1}z,
\end{aligned}\] 
where $k_{j, J}$ is the number of elements $j'\in J$ such that $j' > j$.
Since $g^{|J|+1}=g^{|J|-1}$, this implies the following proposition.
\begin{pr}\label{descriptionofinv}
A (super)polynomial $f=\sum_{l, J}a_{l, J}f_0^l f_1^J$ belongs to $F[V]^{D{g, x}}$ if and only if the following conditions are satisfied.
\begin{enumerate}
\item If $a_{l, J}\neq 0$, then $h^l h^J g^{|J|}=1$,
\item The polynomial 
\[\sum_{l, J}a_{l, J}(\sum_{j\not\in J}(-1)^{k_{j, J\cup j}}l_j x(h_j)f_0^{l-\epsilon_j}f_1^{J\cup j}+\sum_{j\in J}(-1)^{k_{j, J}} f_0^{l+\epsilon_j}f_1^{J\setminus j})\] 
vanishes. 
\end{enumerate}
\end{pr}
We can rewrite the polynomial 
\[\sum_{l, J}a_{l, J}(\sum_{j\not\in J}(-1)^{k_{j, J\cup j}}l_j x(h_j)f_0^{l-\epsilon_j}f_1^{J\cup j}+\sum_{j\in J}(-1)^{k_{j, J}} f_0^{l+\epsilon_j}f_1^{J\setminus j})\] 
from the second condition of the above proposition as 
\[\sum_{l, J} f_0^l f_1^J (\sum_{j\in J}(-1)^{k_{j, J}}(l_j +1)x(h_j)a_{l+\epsilon_j, J\setminus j}+\sum_{j\not\in J}(-1)^{k_{j, J\cup j}} a_{l-\epsilon_j, J\cup j}),\]
where $l_j=0$ implies $a_{l-\epsilon_j, J\cup j}=0$.
\begin{cor}\label{definingequations}
A polynomial $f=\sum_{l, J}a_{l, J}f_0^l f_1^J$ belongs to $F[V]^{D_{g, x}}$ if and only if its coefficients $a_{l,J}$, for all pairs $(l,J)$, satisfy the following equations.
\begin{enumerate}
\item If $h^l h^J g^{|J|}\neq 1$, then $a_{l, J}=0$,
\item $\sum_{j\in J}(-1)^{k_{j, J}}(l_j +1)x(h_j)a_{l+\epsilon_j, J\setminus j}+\sum_{j\not\in J}(-1)^{k_{j, J\cup j}} a_{l-\epsilon_j, J\cup j}=0$.
\end{enumerate}
\end{cor}
If $s=1$, then $F[V]^{D_{g, x}} =F$. Therefore, from now on we will assume that $s> 1$.

Define the partial operator $P_j$ acting on the set of all pairs $(l,J)$ by
$P_j(l, J)=(l+\epsilon_j , J\setminus j)$ in the case when $j\in J$, and $P_j(l, J)$ is undefined if $j\notin J$.
Also define the partial operator $Q_j$ acting on the set of all pairs $(l,J)$ by
$Q_j(l, J)=(l-\epsilon_j, J\cup j)$ in the case $j\not\in J$ and $l_j>0$, and $Q_j(l,J)$ is undefined if $j\in J$ or $l_j=0$.
\begin{lm}\label{commutativity} The operators $P_j$ and $Q_j$ satisfy the following conditions.
\begin{enumerate}
\item If $P_j$ is defined on $(l, J)$, then $Q_j P_j (l, J)=(l, J)$.
Also, if $Q_j$ is defined on $(l, J)$, then $P_j Q_j (l, J)=(l, J)$,
\item If $j\neq j'$ and $P_j Q_{j'}$ is defined on $(l, J)$, then $P_j Q_{j'}(l, J)=Q_{j'}P_j(l, J)$. 
Also, if $j\neq j'$ and $Q_{j'} P_j$ is defined on $(l, J)$, then $Q_{j'}P_j(l, J)=P_j Q_{j'}(l, J)$.
\end{enumerate}
\end{lm}
Two pairs $(l, J)$ and $(l', J')$ are called equivalent if there is a chain 
$(l, J)=(l_0, J_0), \ldots, (l_k, J_k)=(l', J')$
such that $(l_{i+1}, J_{i+1})=S_i(l_i, J_i)$ for $0\leq i\leq k-1$ and each $S_i$ is an operator of type $P$ or $Q$.
Lemma \ref{commutativity} implies that this relation is an equivalence and the set of equations from Corollary \ref{definingequations} is a disjoint union of subsets corresponding to these equivalence classes. 

Moreover, each such equivalence class has a unique representative of the form $(l, \underline{s})$ or $(0, J)$, where the cardinality of $J$ is maximal over this class.
In the first case, all pairs from the equivalence class of $(l, \underline{s})$ can be obtained from this representative by appplying operators of type $Q$ only. 
In the second case, all pairs from the equivalence class of $(0, J)$ can be obtained from $(0, J)$ by applying operators of type $P$ only.
\begin{ex}\label{smallexample}
Let $D=G_m$. Since $X(D)\simeq\mathbb{Z}$, we can fix a generator $h$ of $X=X(D)$. Then $x\in Lie(D)$ is determined by the value $x(h)=\alpha\in F$. 
We will describe invariants of $D_{1,x}$ correposponding to the partial case when $s=2$.

Denote $h_1=h^{k_1}, h_2=h^{k_2}$. 
The subset of equations in Corollary \ref{definingequations} corresponding to the pair $(0, \{1\})$ 
is given as \[\alpha k_1 a_{(1, 0), \emptyset}=0=a_{(0, 0), \{1\}} \]
and the subset corresponding to the pair $(0, \{2\})$ is given as
\[\alpha k_2 a_{(0, 1), \emptyset}=0=a_{(0, 0), \{2\}}. \]

The subset of equations, which corresponds to the pair $((l_1, l_2), \{1, 2\})$, consists of the equations
\[\alpha(-(l_1+1)k_1 a_{(l_1+1, l_2), \{2\}}+(l_2+1)k_2 a_{(l_1, l_2 +1), \{1\}})=0 ,
\]
\[\alpha(l_2+1)k_2 a_{(l_1+1, l_2+1), \emptyset}-a_{(l_1, l_2), \{1, 2\}}=0,\]
\[\alpha(l_1+1)k_1 a_{(l_1+1, l_2+1), \emptyset}+a_{(l_1, l_2), \{1, 2\}}=0
\]
and 
\[a_{(l_1+1, l_2), \{2\}}+a_{(l_1, l_2+1), \{1\}}=0.
\]

If $\alpha=0$ and $k_1, k_2\neq 0$, then the superspace $F[V]^{D_{1, x}}$ is generated by the elements $f_0^{l+\epsilon_1+\epsilon_2}$ and 
$f_0^{l+\epsilon_1}f_1^{\{2\}}-f_0^{l+\epsilon_2}f_1^{\{1\}}$ such that $(l_1+1)k_1+(l_2 +1)k_2=0$.

If $\alpha\neq 0$ and $k_1, k_2\neq 0$, then the superspace $F[V]^{D_{1, x}}$ is generated by the elements 
$f_0^{l+\epsilon_1+\epsilon_2}-\alpha (l_1+1) k_1 f_0^l f_1^{\{1, 2\}}=f_0^{l+\epsilon_1+\epsilon_2}+\alpha (l_2+1) k_2 f_0^l f_1^{\{1, 2\}}$ and 
$f_0^{l+\epsilon_1}f_1^{\{2\}}-f_0^{l+\epsilon_2}f_1^{\{1\}}$ such that $(l_1+1)k_1+(l_2 +1)k_2=0$. 

The remaining cases, when $k_1=0$ or $k_2=0$, are left for the reader.
\end{ex}

Next, let us consider $D_{g, x}$, where $D$ is an arbitrary diagonalizable group and the elements $g$ and $x$ are as above. 
Our aim is to estimate $M_{D_{g, x}, V}$ in terms of the minimal degrees of its diagonalizable (purely even) subsupergroups.

\begin{rem}\label{remrem}
Since $\Delta(g)=g\otimes g$, $F< g>$ is a (purely even) Hopf subsuperalgebra of $F[D_{g, x}]$. 
In other words, there is a short exact sequence of supergroups
\[1\to D'_{1, x}\to D_{g, x}\to\mu_2\to 1,\]
where $F<g>\simeq F[\mu_2]$ and $D'$ is the kernel of the restriction of the epimorphism $D_{g, x}\to\mu_2$. 
Additionally, $F[D']=FX/FX(g-1)$ and $Lie(D)=Lie(D')$.

For every $D_{g, x}$-supermodule $V$ we obtain $F[V]^{D_{g, x}}=(F[V]^{D'_{1, x}})^{\mu_2}$. Therefore $f\in F[V]^{D'_{1, x}}$
implies $f^2\in F[V]^{D_{g, x}}$, which yields $M_{D'_{1, x}, V}\leq M_{D_{g, x}, V}\leq 2M_{D'_{1, x}, V}$. 
\end{rem}

Next, we will consider the special case when $g=1$. Since the element $z$ generates a Hopf supersubalgebra of $F[D_{1, x}]$, there is a supergroup epimorphism $D_{1, x}\to SSp \ K[z] \simeq G_a^-$, where $G_a^-$ is a one-dimensional odd unipotent supergroup. The kernel of this epimorphism coincides with $(D_{1, x})_{ev}\simeq D$. 

Assume that $D_{1, x}$ is connected, which happens if and only if $D$ is connected. Then
$F[V]^{D_{1, x}}=F[V]^{Dist(D_{1, x})}$ (see \cite{zub}).

Since $(D_{1, x})_{ev}$ is (naturally) isomorphic to $D$, from now on we will identify it with $D$. 
The restriction of the comodule map $\tau$ is given by $\tau|_{D}(f_{j, a})=f_{j, a}\otimes h_j$ for $1\leq j\leq s$ and $a=0, 1$. 
We also have $F[V]^{D_{1, x}}=(F[V]^D)^{D_{1, x}/D}$.
\begin{lm}\label{image}
There is a short exact sequence
\[0\to \ I\to Dist(D_{1, x})\to Dist(D_{1, x}/D)\to 0,\] 
where the (two-sided) superideal $I$ is generated by $Dist(D)^+$. 
\end{lm}
\begin{proof}
Since $\mathfrak{m}=F[D_{1, x}]^+=F(X-1)\oplus (FX)z$, we have $\mathfrak{m}^k=\mathfrak{m}_0^k\oplus\mathfrak{m}_0^{k-1} z$.
Therefore $Dist(D_{1, x})=Dist(D)\oplus Dist(D)\phi$, where $\phi$ is an odd element from $Lie(D_{1, x})=(\mathfrak{m}/\mathfrak{m}^2)^*$ such that $\phi(z)=1$ and $\phi(h)=0$ for 
$h\in X$. Since the image of $\phi$, which equals $\phi|_{F[z]}$, generates $Dist(D_{1, x}/D)$, the statement follows.
\end{proof} 
Denote by $\psi$ the restriction $\phi|_{F[z]}$. Then $Dist(D_{1, x}/D)=F\oplus F\psi$ and $\psi^2=0$. 

Let $A$ denote $F[V]^D$. Then $\psi$ acts on $A$ as $\phi|_A$. Furthermore, $\phi$ acts on $F[V]$ as an odd superderivation such that 
$\phi f_{j, 0}=x(h_j) f_{j, 1}$ and $\phi f_{j, 1}=f_{j, 0}$. Hence $F[V]^{D_{1, x}}=A^{Dist(D_{1, x}/D)}=\{a\in A\mid \psi a=\phi a=0\}$. 

Choose a homogeneous basis $\{v_i\}_{i\in I_1\sqcup I_2}$ of the $\mathbb{N}$-graded space $A$ such that the vectors $\{v_i\}_{i\in I_1}$ form a basis of $\phi A$ and the vectors $\{v_i\}_{i\in I_2}$ form a basis of $A/\phi A$. 
Then $\phi a_i=\sum_{j\in I_1} c_{ij} v_j$ for $i\in I_2,$ and the matrix $C=(c_{ij})_{i\in I_2, j\in I_1}$ is row-finite. Since $\phi A\subseteq\ker\phi$, the following Proposition is now evident. 
\begin{pr}\label{onemoreapproachtoinvariants}
The space $F[V]^{D_{1, x}}$ is generated by the vectors $v_i$ for $i\in I_1$, and by the vectors $\sum_{j\in I_2} d_j v_j$, such that the vector 
$d=(d_j)_{j\in I_2}\in F^{I_2}$ satisfies the equation $dC=0$. Moreover, $\phi$ preserves the degrees, which implies $M_{D, V}=M_{D_{1, x}, V}$. 
\end{pr} 
\begin{proof} The first statement is obvious.
For a given monomial $D$-invariant we can create a (non-zero) $D_{1, x}$-invariant of the same degree just by applying the map $\phi$. 
\end{proof}

Returning back to the case of general $g$, using Proposition \ref{onemoreapproachtoinvariants} and Remark \ref{remrem} we derive the following theorem.

\begin{tr}\label{generalcase}
Assume that $D$ is connected and the subgroup $D'$ of $D$ is as in Remark \ref{remrem}. Then for every $D_{g, x}$-supermodule $V$ there are inequalities
$M_{D', V}\leq M_{D_{g, x}, V}\leq 2M_{D', V}$.
\end{tr}
\begin{quest}\label{greatproblem}
Describe all (absolute) invariants of supergroups $D_{g, x}$ assuming that all invariants of $D$ are known.
\end{quest}

\end{document}